\documentclass[12pt,reqno]{amsart}
  
 \usepackage{comment}
\usepackage{tikz}
\usepackage{mathrsfs}
\usepackage{latexsym}
\usepackage{graphicx}
\usepackage{amssymb}
 \graphicspath{{Figures/}}
 \usepackage[show]{ed}
 \usepackage{caption}

\renewcommand{\Re}{{\operatorname{Re}\,}}

\renewcommand{\epsilon}{\varepsilon}

\newcommand{\Vt}{\Vert}

\newcommand{\supp}{{\operatorname{supp\,}}}
\renewcommand{\phi}{\varphi}

\newtheorem{theo}{{\sc Theorem}}

\newtheorem{cor}{{\sc Corollary}}[section]

\newtheorem{lem}[cor]{{\sc Lemma}}

\numberwithin{equation}{section}

\newenvironment{rem}{\medskip\noindent{\it Remark:\/} }{\medskip}

\usepackage{hyperref}

\title[The $L^p$ restriction bounds]
{The $L^p$ restriction bounds for Neumann data on surface}

\author{Xianchao Wu}
\email{xianchao.wu@whut.edu.cn} 

\date{}

\begin{document}

\maketitle
\begin{abstract}
Let $\{u_\lambda\}$ be a sequence of $L^2$-normalized Laplacian eigenfunctions on a compact two-dimensional smooth Riemanniann manifold $(M,g)$. We seek to get an $L^p$ restriction bounds of the Neumann data $ \lambda^{-1} \partial_\nu u_{\lambda}\,\vline_\gamma$ along a unit geodesic $\gamma$. Using the $T$-$T^*$ argument one can transfer the problem to an estimate of the norm of a Fourier integral operator and show that such bound is $O(\lambda^{-\frac{1}p+\frac{3}2})$. The Van De Corput theorem (Lemma \ref{VDC}) plays the crucial role in our proof. Moreover, this upper bound is shown to be optimal.

\end{abstract}

\section{Introduction}
Let $(M,g)$ be a compact smooth Riemannian manifold without boundary. Consider functions satisfying
\begin{equation}
-\Delta_g u_\lambda= \lambda^2 u_\lambda(x), \quad x\in M.
\end{equation}
So, in our case, $\lambda$ is the frequency of the eigenvalue, and also $u_\lambda$ is an $L^2$-normalized eigenfunction with eigenvalue $\lambda$ of $\sqrt{-\Delta_g}$.

Christianson-Hassell-Toth \cite{CHT} and Tacy \cite{Tacy2} showed the boundedness of the Neumann data restricted to a smooth oriented separating hypersurface $\Gamma \subset M$. That is
\begin{equation}\label{known}
\| \lambda^{-1} \partial_\nu u_{\lambda} \|_{L^2(\Gamma)}=O(1),
\end{equation}
where $\nu$ is the normal to the hypersurface $\Gamma$. This estimate can be seen as a statement of non-concentration since the Neumann data $\| \lambda^{-1} \partial_\nu u_{\lambda} \|_{L^2(\Gamma)}$ is only $O(\lambda^{\frac{1}4})$ by \cite{BGT} and from \cite{CHT} we know that \eqref{known} is saturated by considering a sequence of spherical harmonics. We should point out that \eqref{known} is also not explicitly stated in \cite{Tat}.

Note that the proof of \eqref{known} in \cite{CHT} uses Rellich's identity and relies on the result of the restricted $L^2$ upper bound of eigenfunctions in \cite{BGT}. In what follows, we let $\gamma$ denote a geodesic segment of length one and assume that the injectivity radius of $(M,g)$ is 10 or more. In this paper, we aim to give a short direct proof of an $L^p$ restriction bounds of the Neumann data $ \lambda^{-1} \partial_\nu u_{\lambda}\,\vline_\gamma$. Indeed our main theorem is as following,

\begin{theo}\label{main0}
Assume that $(M,g)$ is a compact two-dimensional smooth Riemannian manifold without boundary. There is a $\lambda_0>0$ so that for $p\geq 2$ one has
\begin{equation}\label{main eq}
\| \partial_\nu u_\lambda\|_{L^p(\gamma)} \leq C \lambda^{-\frac{1}p+ \frac{3}2}, \quad \lambda>\lambda_0.
\end{equation}
\end{theo}

\begin{rem}
From the proof of Lemma \ref{general kernel estimate}, one can see that the assumption which $\gamma$ is a geodesic is crucial. However, in a forthcoming paper, we shall use a different method to remove this restriction and prove that \eqref{main eq} holds for any smooth curve.
\end{rem}

This paper is organized in the following way. In Section \ref{section2}, we represent the eigenfunction $u_\lambda$ using the spectrum projector operator and compute its normal derivative along the geodesic $\gamma$. By the $T$-$T^*$ argument, we can transfer our problem \eqref{main eq} to an estimate of the norm of a Fourier integral operator (FIO). The key observation is that one term in the amplitude of the FIO helps us to improve the regularity by Van De Corput theorem (Lemme \ref{VDC}). In Section \ref{section3}, we also prove the optimality of \eqref{main eq}.

\bigskip
\section{Proof of Theorem \ref{main0}}\label{section2}
Using a partition of unity, we assume that $\gamma$ and $\mathcal{T}_{\frac{1}2}(\gamma)$ which is a geodesic tube of width $\frac{1}2$ about $\gamma$ both are contained in one coordinate patch and $x_0=0$. We may work in local coordinates so that $\gamma$ is just
\begin{equation*}
\{(x, 0): -1/2\leq x\leq 1/2\},
\end{equation*}
and we identify the point $x\in \gamma$ as its coordinate $(x, 0)$.

Fix a real-valued even function $\chi\in \mathcal{S}(\mathbb{R})$ having the property that $\chi(0)=1$ and $\hat{\chi}(t)=0$, $|t|\geq\frac{1}2$, where $\hat{\chi}$ denotes the Fourier transform of $\chi$. Then since $\chi_\lambda(u_\lambda):=\chi(\sqrt{-\Delta_g}-\lambda)u_\lambda=u_\lambda$, in order to prove \eqref{main eq}, it suffices to show that
\begin{equation}\label{upshot}
\Vt \partial_\nu \big(\chi(\sqrt{-\Delta_g}-\lambda)u_\lambda\big) \Vt_{L^p(\gamma)}=O(\lambda^{\frac{3}2-\frac{1}p}).
\end{equation}
As shown in \cite{So2017}, the Schwartz kernel of $\chi_\lambda$ takes the form
\begin{equation}
\chi(\sqrt{-\Delta_g}-\lambda)u_\lambda(x)=\lambda^{\frac{1}2} \int_{M} e^{i\lambda \psi(x,y)} a(x,y,\lambda) u_\lambda(y) dy +R_\lambda(u_\lambda)
\end{equation}
where $\psi(x,y)=-d_g(x,y)$ is the geodesic distance with respect to $g$ between $x$ and $y$. Furthermore $a(x,y,\lambda)$ is smooth with all derivatives bounded uniformly in $\lambda$, and is supported where $d_g(x,y) \in [ c_1 \epsilon,c_2 \epsilon]$ for some $c_1, c_2 > 1$. On the other hand, $R_\lambda$ is smooth with all derivatives $O(\lambda^{-N})$ for every $N$.



To compute $\partial_\nu \big(\chi(\sqrt{-\Delta_g}-\lambda)u_\lambda\big)$, $\partial_\nu$ shall act on $e^{i\lambda \psi(x,y)}$, $a(x,y,\lambda)$ and $R_\lambda(u_\lambda)$. To emphasize the variable dependence of $\nu$, we shall use $\nu_x$ instead of $\nu$ in the following. However, the term from $\partial_{\nu_x} a(x,y,\lambda)$ will not bring any troubles to us since we can exactly follow the procedure in \cite{BGT} to get an upper bound $\lambda^{\delta(p)}$, here $\delta(p)=1/2-1/p$ if $4\leq p\leq +\infty$ and   $\delta(p)=1/4$ if $2\leq p\leq 4$. And the term from $\partial_{\nu_x} R_\lambda(u_\lambda)$ is trivially bounded. Hence, the main difficulty comes from $\partial_{\nu_x} e^{i\lambda \psi(x,y)}$.

Notice for points $x$ and $y$, we denote by $e$ the unit vector in $T_xM$ that generates the short geodesic between $x$ and $y$, and write $\theta=\theta(x,y)$ for the angle between the vector $e$ and the normal vector to $\gamma$ at $x$. Then we can express
\begin{equation}
\partial_{\nu_x} \psi_(x,y) =-\sin(\frac{\pi}2-\theta)=-\cos \theta.
\end{equation}
In order to prove \eqref{upshot}, it suffices to show the bound
\begin{equation}
\left\Vt\int_{M}\cos\theta \, e^{i\lambda \psi(x,y)} a(x,y,h) u_\lambda(y) dy \right\Vt_{L^p(\gamma)}=O(\lambda^{-\frac{1}p}).
\end{equation}

We define
\begin{equation}
T_\lambda f(x):= \int_M \cos\theta\, e^{i \lambda \psi(x ,y)} a(x,y, \lambda) f(y) dy.
\end{equation}
It suffices to show that
\begin{equation}
\Vt T_\lambda \Vt_{\mathcal{L}(L^{2}(M); L^{p}(\gamma))} \leq C\lambda^{-\frac{1}p}.
\end{equation}
We shall work in geodesic polar coordinates about $x_0$. So $y=\exp_0(r\omega)$, $r>0$, $\omega\in \mathbb{S}^1$ and $c_1\epsilon\leq r \leq c_2\epsilon$ since the amplitude $a(x,y,\lambda)$ is supported where $d_g(x,y) \in [ c_1 \epsilon,c_2 \epsilon]$ for some $c_1, c_2 > 1$. We can write
\begin{equation}
T_\lambda f (x)= \int_{c_1\epsilon}^{c_2 \epsilon} \int_{\mathbb{S}^1} \cos\theta_{x, \omega} \, e^{i\lambda \psi_r(x,\omega)} a_r(x,\omega, \lambda) f_r(\omega) d\omega dr,
\end{equation}
with
\[\psi_r(x,\omega)=\psi(x,y),\quad f_r(\omega)=f(y), \quad a_r(x,\omega, \lambda)= \kappa(r,\omega) a(x,y, \lambda)\]
for some smooth function $\kappa$, also $\theta_{x, \omega}$ depends on $x$ and the angle $\omega$ of $y$.
Set $T_\lambda^r$ to be the operator with an integral kernel
\begin{equation*}
\cos\theta_{x,\omega} \, e^{i\lambda \psi_r(x,\omega)} a_r(x,\omega, \lambda),
\end{equation*}
hence the integral kernel of $T_\lambda^r (T_\lambda^r)^*$ is
\begin{equation}
K(x, y)=\int_{\mathbb{S}^1} \cos\theta_{x,\omega} \cos\theta_{y,\omega}\, e^{i\lambda [\psi_r(x,\omega)-\psi_r(y,\omega)]}a_r(x, \omega, \lambda) \overline{a_r}(y,\omega, \lambda) d\omega.
\end{equation}
Using $T$-$T^*$ argument, we claim that one only need to show
\begin{equation}\label{upshot1}
\Vt T_\lambda^r \Vt^2_{\mathcal{L}(L^{2}(\mathbb{S}^1); L^{p}(\gamma))}= \|T_\lambda^r (T_\lambda^r)^*\|_{\mathcal{L}(L^{p'}(\gamma); L^p(\gamma))} \leq C\lambda^{-\frac{2}p}.
\end{equation}
Indeed with the help of \eqref{upshot1} we have
\begin{equation*}
\|T_\lambda f\|_{L^p(\gamma)} \leq  \int_{c_1\epsilon}^{c_2 \epsilon} \Vt T_\lambda^r f_r \Vt_{L^p(\gamma)} dr\leq C\lambda^{-\frac{1}p}\int_{c_1\epsilon}^{c_2 \epsilon} \Vt f_r \Vt_{L^p(\gamma)} dr  \leq c\lambda^{-\frac{1}p} \Vt f\Vt_{L^2(M)}.
\end{equation*}

\noindent The main proof is inspired by \cite{BGT} with noticing that the extra term $ \cos\theta_{x,\omega} \cos\theta_{y,\omega}$ shall improve the regularity. In fact following Van De Corput theorem plays the crucial role during the proof.

\begin{lem}\label{VDC}
Suppose that $\Phi(0)$, $\Phi'(0)=0$, and $\Phi'(y)\neq 0$ on $\supp a(\cdot, \lambda)\backslash \{0\}$.  Suppose that $a(\cdot, \lambda)\in C_0^\infty(\mathbb{R})$,
\begin{equation*}
|\partial_y^j a(y, \lambda) | \leq C_j \quad \text{for all}\,\, j
\end{equation*}
and suppose further that $a(0, \lambda)=0$. Then if $\Phi''(0)\neq 0$, one has
\begin{equation}
\Big| \int_0^\infty e^{i \lambda \Phi(y)} a(y, \lambda) dy \Big| \leq C\lambda^{-1}.
\end{equation}
\end{lem}
\begin{proof}
Set 
\[I= \int_0^\infty  e^{i \lambda \Phi(y)} a(y, \lambda) dy= \int_0^1  e^{i \lambda \Phi(y)} a(y, \lambda) dy + \int_1^\infty  e^{i \lambda \Phi(y)} a(y, \lambda) dy.\]
Notice \[ \Big|\int_1^\infty  e^{i \lambda \Phi(y)} a(y, \lambda) dy\Big|\leq C\lambda^{-N}\] by non-stationary phase theorem.
To estimate $ \int_0^1  e^{i \lambda \Phi(y)} a(y, \lambda) dy$, we choose $\rho(y)\in C_0^\infty(\mathbb{R})$ satisfying that
\begin{equation}
\rho=1 \quad\text{when} \,\, |y|<1,\quad\quad \rho=0 \quad\text{when} \,\, |y|>2.
\end{equation}
Then for some small $\delta>0$ which will be determined later, we split the integral
\begin{align*}
I=&\int_0^1 \rho(y/\delta) e^{i \lambda \Phi(y)} a(y, \lambda) dy+\int_0^1 \big(1-\rho(y/\delta) \big) e^{i \lambda \Phi(y)} a(y, \lambda) dy \\
=&I_1+I_2.
\end{align*}
For the first integral, using $|a(y, \lambda)| \leq Cy\,\, \text{if}\,\, y\,\,\text{is closed to}\,\, 0$ we have
\begin{equation*}
|I_1|\leq C \int_0^{2\delta} y dy =C'\delta^2.
\end{equation*}
For the second one, we need to integrate by parts. Let
\begin{equation}
L^*(y, D)=\frac{\partial} {\partial y} \frac{1}{i \lambda \Phi'}.
\end{equation}
Then for $N=0,1,2, \dots$
\begin{align}\label{I2}
|I_2|=&\Big| \int_0^1  e^{i \lambda \Phi(y)} \big(L^*(y, D)\big)^N \Big(  \big(1-\rho(y/\delta) \big) a(y, \lambda) \Big) dy \Big| \nonumber \\
 \leq& \int_{\delta}^1 \Big| \big(L^*(y, D)\big)^N \Big(  \big(1-\rho(y/\delta) \big) a(y, \lambda) \Big) \Big| dy
\end{align}
Notice
\[\Big|\frac{1}{\Phi'(y)}\Big| \leq C\frac{1}y \quad \text{and}\,\, |a(y, \lambda)| \leq Cy, \]
and using the product rule for differentiation implies that the last integrand in \eqref{I2} is majorized by 
\begin{equation*}
\lambda^{-N} y^{1-N}\delta^{-N}
\end{equation*}
for sufficiently large $N$ with noticing that $y\geq \delta$. Hence we can dominate $|I_2|$ by
\[C\lambda^{-N}\delta^{2-2N}.\]
In conclusion, we get
\begin{equation*}
|I| \leq C\big(\delta^2 + \lambda^{-N}\delta^{2-2N}\big).
\end{equation*}
The right side would achieve its smallest when the two summands agree, that is $\delta=\lambda^{-\frac{1}2}$, which gives
\begin{equation*}
|I|\leq C\lambda^{-1}
\end{equation*}
as desiered.

\end{proof}

\noindent Then we can estimate the kernel $K(x,y)$.
\begin{lem}\label{general kernel estimate}
There exists $\delta>0$ such that for $|x-y|<\delta$, here $x,y\in \gamma$, one has
\begin{equation}
|K(x,y)| \lesssim (1+\lambda |x-y|)^{-1}.
\end{equation}
\end{lem}
\begin{proof}
Recall that 
\begin{equation*}
K(x, y)=\int_{\mathbb{S}^1} \cos\theta_{x,\omega} \cos\theta_{y,\omega}\, e^{i\lambda [\psi_r(x,\omega)-\psi_r(y,\omega)]}a_r(x, \omega, \lambda) \overline{a_r}(y,\omega, \lambda) d\omega.
\end{equation*}
Applying Toylor's formula, one has
\begin{equation}
\psi_r(x,\omega)-\psi_r(y ,\omega)=\left< x-y, \Psi(x, y, \omega) \right>=|x-y|\left< \sigma, \Psi(x, y, \omega) \right>,
\end{equation}
here 
\begin{equation*}\sigma=\frac{x- y}{|x- y|} \in \mathbb{S}^1 \quad\text{and}\quad \Psi(x, y, \omega) =\int_0^1 \nabla_x \psi_r(y+t(x-y), \omega) dt.
\end{equation*}

Set
\begin{equation*}
\Phi(x, y, \sigma, \omega)=\left< \sigma, \Psi(x, y, w) \right>,
\end{equation*}
also we parametrize the circle $\mathbb{S}^1$ as
\begin{equation*}
\omega=\omega(w)=(\cos w, \sin w), \quad w\in[0, 2\pi],
\end{equation*}
and suppose that
\begin{equation*}
\sigma=(\cos \alpha, \sin \alpha), \quad \alpha\in [0,2\pi].
\end{equation*}
So we have
\begin{equation*}
K(x, y)= \int_{0}^{2\pi}  \cos\theta_{x,\omega(w)} \cos\theta_{y,\omega(w)}\, e^{i\lambda |x-y|\Phi(x, y, \sigma, \omega(w))}a_r(x, \omega(w), \lambda) \overline{a_r}(y,\omega(w), \lambda) dw.
\end{equation*}
Now we split  the integral kernel $K(x,y)$ into 4 parts according to disjoint 4 intervals $I_1(\sigma)$, $I_2(\sigma)$, $I_3(\sigma)$ and $I_4(\sigma)$ which shall be specified later,
\begin{equation}
K(x,y)=K_{I_1}(x,y)+K_{I_2}(x,y)+K_{I_3}(x,y)+K_{I_4}(x,y),
\end{equation}
where
\begin{align*}
K_{I_j}(x,y)=\int_{I_j(\sigma)}  \cos\theta_{x,\omega(w)} \cos\theta_{y,\omega(w)}\, e^{i\lambda |x-y|\Phi(x, y, \sigma, \omega(w))}a_r(x, \omega(w), \lambda) \overline{a_r}(y,\omega(w), \lambda) dw, \\
\quad j=1,2,3,4.
\end{align*}

Next we deal with the phase function $\Phi(x, y, \sigma, \omega(w))$.
Thanks to \cite[Lemma 3.1]{BGT}, one has
\begin{equation*}
\Psi(0, 0, \omega)=\big(\partial_{x_1} \psi_r(0, \omega), \partial_{x_2} \psi_r(0, \omega)\big)=\omega=(\cos w, \sin w).
\end{equation*}
Hence
\begin{equation*}
\Phi(0,0, \sigma, \omega)=\left< \sigma, \Psi(0, 0, w) \right>=\left<\sigma, \omega \right>=\cos (w-\alpha).
\end{equation*}

Compute
\begin{equation}
\Phi'_w(0,0, \sigma, \omega(w))=-\sin(w-\alpha)
\end{equation}
and 
\begin{equation}
\Phi''_{ww}(0,0, \sigma, \omega(w))=-\cos(w-\alpha).
\end{equation}
Notice when $w$ equals $\alpha$ or $\alpha+\pi$, $|\Phi'_w(0,0, \sigma, \omega(w))|=0$ and $|\Phi''_{ww}(0,0, \sigma, \omega(w))|=1$. So we can split $[0, 2\pi]$ into disjoint 4 intervals $I_1(\sigma)$, $I_2(\sigma)$, $I_3(\sigma)$ and $I_4(\sigma)$, where $I_2(\sigma)$ and $I_4(\sigma)$ are neighborhoods of $\alpha$ and $\alpha+\pi$ respectively, such that
\begin{equation*}
|\Phi'_w(0,0, \sigma, \omega(w))|\geq \frac{\sqrt{2}}2,  \,\, |\Phi''_w(0,0, \sigma, \omega(w))|\leq \frac{\sqrt{2}}2, \quad w\in I_1(\sigma)\cup I_3(\sigma)
\end{equation*}
and
\begin{equation*}
|\Phi'_w(0,0, \sigma, \omega(w))|\leq \frac{\sqrt{2}}2, \,\, |\Phi''_w(0,0, \sigma, \omega(w))|\geq \frac{\sqrt{2}}2,\quad w\in I_2(\sigma)\cup I_4(\sigma).
\end{equation*}
By continuity, one has
\begin{equation} \label{1and3}
|\Phi'_w(x,y, \sigma, \omega(w))|\geq \frac{\sqrt{2}}4,\quad w\in I_1(\sigma)\cup I_3(\sigma)
\end{equation}
and
\begin{equation} \label{2and4}
|\Phi''_w(x,y, \sigma, \omega(w))|\geq \frac{\sqrt{2}}4,\quad w\in I_2(\sigma)\cup I_4(\sigma).
\end{equation}

Thanks to the noncritical condition \eqref{1and3}, one can integrate by parts to show that $K_{I_1}(x,y)$ and $K_{I_3}(x,y)$ are bounded by
\[C(1+\lambda|x-y|)^{-1}.\]

\begin{figure}
\begin{tikzpicture}
\draw[thick] (-4,0) -- (0,0) node[below]{$x_0$} -- (1,0) node[below]{$x$} -- (2,0) node[below]{$y$} -- (4,0) node[below]{$\gamma$};
\draw[dashed,thick] (1,0) -- (1,1.5);
\draw[dashed,thick] (2,0) -- (2,1.5);
\draw[dashed,thick] (0,0)--(2.5, 1.25);
\draw[dashed,thick] (1,0)--(2.5, 1.25);
\draw[dashed,thick] (2,0)--(2.5, 1.25);
\draw[fill] (0,0) circle [radius=1pt];
\draw[fill] (1,0) circle [radius=1pt];
\draw[fill] (2,0) circle [radius=1pt];
\draw[thick] (1.3, 0.3) arc (30: 90: 0.3);
\draw[thick] (2.2, 0.5) arc (60: 90: 0.4);
\node at (1.3, 0.6) {$\theta_{x,w}$};
\node[right] at (2.2, 0.5)  {$\theta_{y,w}$};
\node[right] at (2.5, 1.25) {$\exp_0(r\omega)$};
\end{tikzpicture}
\caption{}
\label{fig}
\end{figure}

Next we shall deal with the bounds of  $K_{I_2}(x,y)$ and $K_{I_4}(x,y)$. The key observation is that $\cos\theta_{x,\omega} \cos\theta_{y,\omega}=0$ when $w=\alpha$ or $w=\alpha+\pi$ due to the fact that $x$ and $y$ both are on the geodesic $\gamma$ (Figure \ref{fig}). When $\lambda^{-1}\leq |x-y|< \delta$, thanks to \eqref{2and4} and Lemma \ref{VDC}, $K_{I_2}(x,y)$ and $K_{I_4}(x,y)$ are bounded by
\[C\lambda^{-1}|x-y|^{-1} \leq C'(1+\lambda|x-y|)^{-1}.\]
On the other hand, when $|x-y|\leq \lambda^{-1}$, they are trivially bounded by $O(1)$ which is also less than $C(1+\lambda|x-y|)^{-1}$.
These complete the proof.
\end{proof}

\begin{rem}
From above argument, the assumption that $\gamma$ is geodesic plays an essential role.
\end{rem}

Using Lemma \ref{general kernel estimate} and Young's inequality,
\begin{align*}
\|T_\lambda^r(T_\lambda^r)^* f\|_{L^p(\gamma)} &\lesssim \| \int_{-\frac{1}2}^{\frac{1}2} (1+\lambda|x-y|)^{-1} f(y)) dy  \|_{L^p(\gamma)} \\
&\lesssim \|(1+\lambda|x|)^{-1}\|_{L^{\frac{p}2}({-\frac{1}2}, {\frac{1}2})} \|f\|_{L^{p'}(\gamma)}.
\end{align*}
For $p>2$
\begin{equation}
\|(1+\lambda|x|)^{-1}\|_{L^{\frac{p}{2} }({-\frac{1}2},{\frac{1}2})}\leq \lambda^{-\frac{2}p} \Big( \int_0^{C\lambda} (1+\tau)^{-\frac{p}2} d\tau \Big)^{\frac{2}p} \lesssim \lambda^{-\frac{2}p}.
\end{equation}
For $p=2$, notice that from the proof of Lemma \ref{general kernel estimate}, for $|x-y|\geq \lambda^{-1}$,
\[|K(x,y)| \lesssim (1+ \lambda |x-y|)^{-1},\]
and for $|x-y|\leq \lambda^{-1}$, $|K(x,y)|$ is $O(1)$.

Applying Young's inequality, one can get
\begin{align*}
\|T_\lambda^r(T_\lambda^r)^* f\|_{L^2(\gamma)} \leq \Big| \int_{|x-y|\geq \lambda^{-1}} K(x,y) f(y) dy \Big|+\Big| \int_{|x-y|\leq \lambda^{-1}} K(x,y) f(y) dy \Big| \\
\lesssim \lambda^{-1} \|f\|_{L^{2}(\gamma)}.
\end{align*}
Hence one can get \eqref{upshot1}.

\bigskip
\section{The optimality of \eqref{main eq}} \label{section3}
In this section we prove lower bounds for the operator $\mathcal{T}:=\partial_\nu \chi(\sqrt{-\Delta_g}-\lambda)$ and we will use this bound to get the optimality of \eqref{main eq}. The main proof is motivated by \cite{BGT}.

\begin{lem}
There exists $c>0$ such that for any $\lambda\geq 1$,
\begin{equation}
\|\mathcal{T}\|_{\mathcal{L}(L^2(M); L^p(\gamma)) }\geq c\lambda^{\frac{3}2-\frac{1}p}.
\end{equation}
\end{lem}
\begin{proof}
From the discussion of Section \ref{section2}
it suffices to show
\begin{equation}
\|T_\lambda^r(T_\lambda^r)^*\|_{\mathcal{L}( L^{p'}(\gamma); L^p(\gamma) )} \geq c\lambda^{-\frac{2}p}.
\end{equation}
Recall that the integral kernel of $T_\lambda^r (T_\lambda^r)^*$ is
\begin{equation*}
K(x, y)=\int_{0}^{2\pi}  \cos\theta_{x,\omega(w)} \cos\theta_{y,\omega(w)}\, e^{i\lambda |x-y|\Phi(x, y, \sigma, \omega(w))}a_r(x, \omega(w), \lambda) \overline{a_r}(y,\omega(w), \lambda) dw.
\end{equation*}
If $|x-y|\leq 2\epsilon \lambda^{-1}$, the oscillatory factor $e^{i\lambda |x-y|\Phi(x, y, \sigma, \omega(w))}$ does not oscillate any more.
Let
\begin{align}
\Re(K_1(x,y))& \nonumber \\
:=\int_{0}^{2\pi}  \cos&\theta_{x,\omega(w)} \cos\theta_{y,\omega(w)} \chi_\alpha(w)\, \Re\Big(e^{i\lambda |x-y|\Phi(x, y, \sigma, \omega(w))}\Big)a_r(x, \omega(w), \lambda) \overline{a_r}(y,\omega(w), \lambda) dw,
\end{align}
here $\chi_\alpha(w)$ is a cut off function such that $0\leq\max\{\theta_{x,\omega},\theta_{y,\omega}\}\leq \pi/4$ (Figure \ref{fig}). So for sufficiently small $\epsilon$, one has $\Re(K_1(x,y))\geq\beta>0$, for some constant $\beta$.
As a consequence, choosing a test function $f(y)=\lambda^{\frac{1}{p'}}\phi(\frac{\lambda y}{\epsilon})$, where $\phi\in C_0^\infty(-1,1)$, $\phi\geq0$, $\phi=1$ on $[-\frac{1}2, \frac{1}2]$ and $\frac{1}p+\frac{1}{p'}=1$, we obtain (for a constant $c>0$)
\begin{equation}
\Re\big(1_{|x|\leq \epsilon \lambda^{-1}} T_\lambda^r (T_\lambda^r)^* f\big) = 1_{|x|\leq \epsilon \lambda^{-1}} \int \Re(K_1(x, y)) f(y) dy
\geq c 1_{|x|\leq \epsilon \lambda^{-1}} \lambda^{\frac{1}{p'}-1}.
\end{equation}
Hence
\begin{equation}
\|T_\lambda^r (T_\lambda^r)^* f \|_{L^p}\geq \|T_\lambda^r (T_\lambda^r)^* f \|_{L^p(|x|\leq \epsilon \lambda^{-1})}  \geq \|\Re\big(1_{|x|\leq \epsilon \lambda^{-1}} T_\lambda^r (T_\lambda^r)^* f\big)\|_{L^p} \geq c \lambda^{\frac{1}{p'}-1-\frac{1}p}=c\lambda^{-\frac{2}p}.
\end{equation}

\end{proof}

We next show that the lower bound for the norm of the operator $\mathcal{T}$ from $L^2(M)$ to $L^p(\gamma)$ gives a lower bound for the operator
\begin{equation*}
\partial_\nu\Pi_\lambda= \partial_\nu 1_{\sqrt{-\Delta_g}-\lambda\in [0,\frac{1}2)} 
\end{equation*}
in the following sense:
\begin{equation}\label{subseq}
\exists\, c>0,\,\, \exists\, \lambda_n\to +\infty,\,\, f_n\in L^2(M) \quad\text{such that}\,\, \|\partial_\nu (\Pi_{\lambda_n} f_n) \|_{L^p(\gamma)}> c\lambda_n^{\frac{3}2-\frac{1}p} \|f_n\|_{L^2(M)}.
\end{equation}
Indeed, if the converse were true, there exists an $f\in L^2(M)$ such that
\begin{equation}\label{counter}
\|\partial_\nu(\Pi_{\lambda} f)\|_{L^p(\gamma)}< \epsilon(\lambda)\lambda^{\frac{3}2-\frac{1}p} \|f\|_{L^2(M)},\quad \text{with} \lim_{\lambda\to +\infty} \epsilon(\lambda)=0.
\end{equation}
Using a partition of unity
\begin{equation*}
1=\sum_{n\in \mathbb{Z}} \Pi_{\frac{n}2}
\end{equation*}
and inserting another projector $\tilde\Pi_\lambda= 1_{\sqrt{-\Delta_g}-\lambda\in [-1,1]} $, we obtain
\begin{equation}
\chi(\sqrt{-\Delta_g}-\lambda)=\sum_{n\in\mathbb{Z}} \Pi_{\frac{n}2}\tilde\Pi_{\frac{n}2} \chi(\sqrt{-\Delta_g}-\lambda),
\end{equation}
with noticing that $ \Pi_{\frac{n}2}\tilde\Pi_{\frac{n}2}= \Pi_{\frac{n}2}$.

By \eqref{counter} one has
\begin{equation*}
\|\partial_\nu \chi(\sqrt{-\Delta_g}-\lambda)\|_{\mathcal{L}(L^{2}(M); L^p(\gamma))} \leq \sum_n\epsilon(n) n^{\frac{3}2-\frac{1}p} \|\tilde\Pi_{\frac{n}2} \chi(\sqrt{-\Delta_g}-\lambda)\|_{\mathcal{L}(L^{2}(M)}.
\end{equation*}
But due to the rapid decay of the function $\chi$,
\begin{equation*}
\|\tilde\Pi_{\frac{n}2} \chi(\sqrt{-\Delta_g}-\lambda)\|_{\mathcal{L}(L^{2}(M)}\leq \frac{C_N}{(1+|\lambda-\frac{n}2|)^N}
\end{equation*}
and consequently we obtain
\begin{equation}
\|\partial_\nu \chi(\sqrt{-\Delta_g}-\lambda)\|_{\mathcal{L}(L^{2}(M); L^p(\gamma))} \leq o(1)\lambda^{\frac{3}2-\frac{1}p} \quad \text{as}\,\, \lambda\to +\infty,
\end{equation}
which is contradicting our lower bounds for the operator $\mathcal{T}$.

Now, if the manifold $M$ is a sphere and $\{\lambda_n\}$ is the sequence satisfying \eqref{subseq}, then  sequence $g_n=\Pi_{\lambda_n}f_n$ satisfies the claimed lower bound. Here the range of projector $\Pi_\lambda$ is the finite dimensional space of spherical harmonics of degree $k$ with eigenvalues $-k(k+1)$ if $\sqrt{k(k+1)}-\lambda\in[0,\frac{1}2)$.

\begin{rem} \cite{CHT} shows the optimality of \eqref{main eq} for $p=2$ by considering the highest weight spherical harmonics on the longitude. For $p=\infty$, \eqref{main eq} is optimized on a geodesic passing through the north pole with noticing that the zonal spherical harmonics have the supremum of $O(\lambda^{1/2})$ at the north pole and concentrate in $O(\lambda^{-1})$ neighborhoods of the north pole.
\end{rem}

\bibliography{reference}
\bibliographystyle{alpha}

\end{document}